\documentclass[12pt,a4paper]{article}
\usepackage{amsmath}
\usepackage{amssymb}
\usepackage{amscd}
\usepackage{graphics}
\usepackage{color}
\usepackage[all]{xy}
\usepackage{amsthm}   
\usepackage{mathrsfs}
\usepackage{setspace}
\numberwithin{equation}{section}
\theoremstyle{plain}
\newtheorem{theorem}{Theorem}[section]
\newtheorem{lemma}[theorem]{Lemma}
\newtheorem{corollary}[theorem]{Corollary}
\newtheorem{proposition}[theorem]{Proposition}
\theoremstyle{definition}

\newtheorem{example}[theorem]{Example}
\theoremstyle{remark}
\newtheorem{remark}[theorem]{Remark}
\newcommand{\RMOD}{\ensuremath{{R\text{-}\!\mathcal{M}od}}} 
\newcommand{\Image}{\operatorname{Im}} 
\newcommand{\Ker}{\operatorname{Ker}}
\newcommand{\Rad}{\operatorname{Rad}}

\newcommand{\ShortExactSequence}[5]{\ensuremath{\xymatrix@1{ 0 \ar[r] &  #1
\ar[r]^-{#2} & #3 \ar[r]^-{#4} &  #5
\ar[r] & 0 }}}

\newcommand{\End}{\operatorname{End}}
 \newcommand{\Ext}{\operatorname{Ext}} 
 \newcommand{\To}{\longrightarrow}   

\begin{document}
\pagestyle{myheadings}
\enlargethispage{\baselineskip}
\title{\vspace{-2cm} Rad-supplementing modules}

\author{Salahattin \"{O}zdemir\footnote{
            Department of Mathematics, Faculty of Sciences, Dokuz Eyl\"{u}l
University, Turkey.$\quad$
            e-mail: \texttt{salahattin.ozdemir@deu.edu.tr}.
              }
}

\maketitle
\renewcommand{\theenumi}{\roman{enumi}}
\renewcommand{\labelenumi}{\emph{(\theenumi)}}

\begin{abstract}
Let $R$ be a ring and  $M$ be a left $R$-module. If $M$ is
Rad-supplementing, then every direct summand of $M$ is
Rad-supplementing, but not each factor module of $M$. Any finite direct sum of Rad-supplementing modules
is Rad-supplementing. Every module with composition series is (Rad-)supplementing. $M$ has a Rad-supplement in its injective envelope if and only if $M$ has a Rad-supplement in every essential extension. $R$ is left perfect if and only if $R$ is semilocal, reduced and the free left $R$-module
$(_R R)^{(\mathbb{N})}$ is Rad-supplementing if and only if $R$ is reduced and the free left $R$-module
$(_R R)^{(\mathbb{N})}$ is ample Rad-supplementing. $M$ is ample Rad-supplementing if and only if every submodule of $M$ is Rad-supplementing. Every left $R$-module is (ample) Rad-supplementing if and only if $R/P(R)$ is left perfect, where $P(R)$ is the sum of all left ideals $I$ of $R$ such that $\Rad I = I$.
\end{abstract}

\textbf{$\quad$Key words:} supplement, Rad-supplement, supplementing module,

$\quad$ Rad-supplementing module, perfect ring.

\section{Introduction}\label{sec:introduction}

All rings consider in this paper will be associative with an identity
element. Unless otherwise stated, $R$ denotes an arbitrary ring and all modules will be \emph{left}
unitary $R$-modules. By $\RMOD$ we denote the category of left $R$-modules. Let $M$ be a module. By $X\subseteq M$, we mean $X$ is a submodule of $M$ or $M$ is an extension of $X$. As usual, $\Rad M$ denotes the radical of $M$ and $J$ denotes the jacobson radical of the ring $R$. $E(M)$ will be the injective envelope of  $M$.  For an index set $I$, $M^{(I)}$  denotes as usual the direct sum $\oplus_{I} M$. By $\mathbb{N}$, $\mathbb{Z}$ and $\mathbb{Q}$ we denote as usual  the set of natural numbers, the ring of integers and the field of rational numbers, respectively.
  A submodule $K \subseteq M$ is called
\emph{small} in $M$ (denoted by $K\ll M$) if $M\neq K+T$ for every
proper submodule $T$ of $M$. Dually, a submodule $L\subseteq M$ is called \emph{essential} in $M$ (denoted by $L\trianglelefteq M$) if $L\cap X \neq 0$ for every nonzero submodule $X$ of $M$.

The notion of a supplement submodule was introduced in \cite{Kasch-Mares} in order to characterize semiperfect modules, that is projective modules whose factor modules have projective cover. For submodules $U$ and $V$ of a module $M$, $V$ is said to be a \emph{supplement} of $U$ in $M$ or $U$ is said to \emph{have a supplement} $V$ in $M$ if $U+V = M$ and $U\cap V \ll V$. $M$ is called a \emph{supplemented} module if every submodule of $M$ has a supplement in $M$. See
\cite[\S 41]{Wisbauer:FoundationsofModuleandRingTheory} and \cite{LiftingModules:Wisbauer-et.al} for results (and the definitions) related to supplements and supplemented modules. Recently, several authors have studied different generalizations of supplemented modules. In \cite{Wisbauer-et.al:t-complementedandt-supplementedmodules}, $\tau$-supplemented modules were defined for an arbitrary preradical $\tau$ for $\RMOD$. For submodules $U$ and $V$ of a module $M$, $V$ is said to be a $\tau$-\emph{supplement} of $U$ in $M$ or $U$ is said to \emph{have a} $\tau$-\emph{supplement} $V$ in $M$ if $U+V = M$ and $U\cap V \subseteq \tau(V)$. $M$ is called a $\tau$-\emph{supplemented} module if every submodule of $M$ has a $\tau$-supplement in $M$.
For the particular case $\tau = \Rad$, Rad-supplemented modules have been studied in \cite{Rad-supplementedModules}; rings over which all modules are Rad-supplemented were characterized. See \cite{WangY:GeneralizedSupplementedModules}; these modules are called \emph{generalized supplemented} modules. Note that Rad-supplements $V$ of a module $M$ are also called \emph{coneat} submodules, and can be characterized by the fact that each module with zero radical is injective with respect to the inclusion $V\subseteq M$; see \cite[\S 10]{LiftingModules:Wisbauer-et.al}, \cite{Wisbauer-et.al:t-complementedandt-supplementedmodules} and \cite{Mermut:Ph.D.tezi}. On the other hand, modules that have supplements in every module in which it is contained as a submodule have been studied in \cite{Zoschinger:ModulnDieInJederErweiterungEinKomplementHaben}; the structure of these modules, which are called \emph{modules with the property (E)}, has been completely determined over Dedekind domains. Such modules are also called  \emph{ Moduln mit Erg\"{a}nzungseigenschaft} in \cite{Averdunk:Ph.D.thesis} and \emph{supplementing} modules in \cite[p.255]{LiftingModules:Wisbauer-et.al}. Also, in the recent paper \cite{Turkmen-Calisici:modules-that-have-supplements-in-every-cofinite-extansion},  modules that have a supplement in every cofinite extension have been studied, where a module $N$ is called a \emph{cofinite extension} of $M$ if $M\subseteq N$ and $N/M$ is finitely generated; see \cite{AlizadeAndBilhanAndSmith:ModulesWhoseMaximalSubmodulesHaveSupplements} for the notion of a cofinite submodule.  We follow the terminology and notation as in \cite{LiftingModules:Wisbauer-et.al}.  We call a module $M$ \emph{supplementing} if it has a supplement in each module in which it is contained as a submodule. By considering these modules we define and study (ample) Rad-supplementing modules as a proper generalization of supplementing modules. A module $M$ is called \emph{(ample) Rad-supplementing} if it has a (an ample) Rad-supplement in each module in which it is contained as a submodule, where a submodule $U \subseteq M$ has \emph{ample Rad-supplements} in
$M$ if for every $L\subseteq M$ with $U + L = M$, there is a
Rad-supplement $L'$ of $U$ with $L' \subseteq L$.

In section \ref{sec:Rad-supplementing modules} we investigate some properties of Rad-supplementing modules. It is clear that every supplementing module is Rad-supplementing, but the converse implication fails to be true, Example \ref{exm}. If a module $M$ has a Rad-supplement in its injective envelope, $M$ need not be Rad-supplementing. However, we prove that $M$ has a Rad-supplement in its injective envelope if and only if $M$ has a Rad-supplement in every essential extension, Proposition \ref{prop:M has Rad-supp.in N iff it has in E(M) iff it has in E}. Using the fact that, for modules $A\subseteq B$, if $A$ and $B/A$ are (Rad-)supplementing then so is $B$, we prove that every module with composition series is (Rad-)supplementing, Theorem \ref{thm:every-module-with-composition-series-is-rad-supp.}. A factor module of a Rad-supplementing module need not be Rad-supplementing, Example \ref{exm2}. For modules $A\subseteq B \subseteq C$ with $C/A$ injective, we prove that if $B$ is Rad-supplementing then so is $B/A$. As one of the main results, we prove that $R$ is left perfect if and only if $R$ is semilocal, $_R R$ is reduced and $(_R R)^{(\mathbb{N})}$ is Rad-supplementing. Finally, using a result of \cite{Zoschinger:ModulnDieInJederErweiterungEinKomplementHaben}, we show that over a commutative noetherian ring $R$, a semisimple $R$-module $M$ is Rad-supplementing if and only if it is supplementing and that is equivalent the fact that $M$ is pure-injective, Theorem \ref{thm:supplementing=rad-supplementing=pure-injective}.

Section \ref{sec:Ample-rad-supp} contains some properties of ample Rad-supplementing modules.  It starts by proving a useful property that a module $M$ is ample Rad-supplementing if and only if every submodule of $M$ is Rad-supplementing, Proposition \ref{prop:M-ample-rad-supplementing-iff-every-submodule-rad-supplementing}.  One of the main results of this part is that $R$ is left perfect if and only if $_R R$ is reduced and the free left $R$-module $(_R R)^{(\mathbb{N})}$ is ample Rad-supplementing, Theorem \ref{thm:R-perfect-iff-R-reduced-and-F-ample-rad-supp}. In the proof of this result, Rad-supplemented modules plays an important role as, of course,  every ample Rad-supplementing module is Rad-supplemented. Finally, using the characterization of Rad-supplemented modules given in \cite{Rad-supplementedModules}, we characterize the rings over which every module is (ample) Rad-supplementing. We prove that every left $R$-module is (ample) Rad-supplementing if and only if every reduced left $R$-module is Rad-supplementing if and only if $R/P(R)$ is left perfect, Theorem \ref{thm:characterizaiton-of-rad-supp}.

\section{Rad-supplementing modules}\label{sec:Rad-supplementing modules}

A module $M$ is called  \emph{radical} if $\Rad M = M$, and $M$ is called \emph{reduced} if it has no nonzero radical submodule. See \cite[p.47]{Zoschinger:KomplementierteModulnUberDedekindringen} for details for the notion of reduced and radical modules.

\begin{proposition}\label{prop:radical-modules-are-rad-supp.}
Supplementing modules and radical modules are Rad-supplementing.
\end{proposition}

\begin{proof}
Let $M$ be a module and $N$ be any extension of $M$. If $M$ is supplementing, then it has a supplement, and so a Rad-supplement in $N$. Thus $M$ is Rad-supplementing. Now, If $\Rad M = M$, then $N$ is a Rad-supplement of $M$ in $N$.
\end{proof}

 By $P(M)$ we denote the sum of all \emph{radical}
submodules of the module $M$, that is, $$P(M)=\sum \{ U\subseteq M \mid \,
\Rad U=U\} .$$ Clearly $M$ is reduced if $P(M)=0$.

Since $P(M)$ is a radical submodule of $M$ we have the following corollary.

\begin{corollary}\label{cor:P(M)-Rad-supplementing}
For a module $M$, $P(M)$ is Rad-supplementing.
\end{corollary}

Recall that a subset $I$ of a ring $R$ is said to be \emph{left} $T$\emph{-nilpotent} in case, for every sequence $\{a_k\}_{k=1}^{\infty}$ in $I$, there is a positive integer $n$ such that $a_1 \cdots a_n =0$.

In general, Rad-supplementing modules need not be supplementing as the following example shows.
\begin{example}\label{exm}
Let $k$ be a field.
 In the polynomial ring
$ k[x_{1},x_{2},\ldots]$ with countably many indeterminates $x_n$, $n\in \mathbb{N}$, consider the ideal $I=(x_{1}^2, x_{2}^2 -x_{1},x_{3}^2 -x_{2},\ldots)$  generated by
 $x_{1}^2$ and $x_{n+1}^2 -x_{n}$ for each $n\in \mathbb{N}$. Then the quotient ring $R=k[x_{1},x_{2},\ldots]/I$ is a local ring with  the unique maximal ideal $J=J^2$ (see \cite[Example 6.2]{Rad-supplementedModules} for details). Now let $M = J^{(\mathbb{N})}$. Then we have $\Rad M = M$, and so $M$ is Rad-supplementing by Proposition \ref{prop:radical-modules-are-rad-supp.}. However, $M$ does not have a supplement in $R^{(\mathbb{N})}$. Because, otherwise,  by \cite[Theorem 1]{Engin-Lomp:Rings-Whose-modules-are-Weakly-supplemented-are-perfect}, $J$ would be a left $T$-nilpotent as $R$ is semilocal,  but this is impossible. Thus $M$ is not supplementing.
\end{example}

For instance, over a left max ring, supplementing modules and Rad-supplementing modules coincide, where $R$ is called a \emph{left max ring} if every left $R$-module has a maximal submodule or equivalently, $\Rad M \ll M$ for every left $R$-module $M$.

\begin{proposition}
Every direct summand of a Rad-supplementing module is
Rad-supplementing.
\end{proposition}

\begin{proof}
Let $M$ be a Rad-supplementing module, $U$ be a direct summand of $M$ and let $N$ be any extension of $U$. Then
$M=A\oplus U$ for some submodule $A\subseteq M$. By hypothesis
$M$ has a Rad-supplement in the module $A \oplus N$ containing $M$,
that is, there exists a submodule $V$ of $A\oplus N$ such that
$$(A \oplus U) + V = A \oplus N \quad\text{ and }\quad (A \oplus U)
\cap V \subseteq \Rad V .$$ Let $g:A\oplus N \rightarrow N$ be the
projection onto $N$.
Then $$U + g(V) = g(A \oplus U) + g(V) = g((A\oplus U) + V) =
g(A\oplus N) = N \text{ and },$$
$$U \cap g(V) = g((A\oplus U)\cap V)\subseteq g(\Rad
V)\subseteq \Rad (g(V)).$$
Hence $g(V)$ is a Rad-supplement of $U$ in $N$.
\end{proof}

If a module $M$ has a Rad-supplement in its injective envelope $E(M)$, $M$ need not be Rad-supplementing. For example, for $R = \mathbb{Z}$, the $R$-module $M=2\mathbb{Z}$ has a Rad-supplement in $E(M) = \mathbb{Q}$ since $\Rad \mathbb{Q} = \mathbb{Q}$ (and so $\mathbb{Q}$ is Rad-supplemented). But, $M$ does not have a Rad-supplement in $\mathbb{Z}$, and thus $M$ is not Rad-supplementing. However, we have the following result.

\begin{proposition}\label{prop:M has Rad-supp.in N iff it has in E(M) iff it has in E}
Let $M$ be a module. Then the following are equivalent.
\begin{enumerate}
 \item[(i)] $M$ has a Rad-supplement in every essential extension,
 \item[(ii)] $M$ has a Rad-supplement in its injective envelope $E(M)$.
\end{enumerate}
\end{proposition}

\begin{proof}
$(i) \Rightarrow (ii)$ is clear.
$(ii) \Rightarrow (i)$ Let $M \subseteq N$ with $M\trianglelefteq N$, and let $f: M\to N$ and
$g: M\to E(M)$ be inclusion maps. Then we have the following
commutative diagram with $h$ necessarily monic:
$$
\xymatrix{ M \ar@{^{(}->}[r]^f  \ar@{^{(}->}[d]_g & N \ar@{-->}[dl]^h \\
 E(M)  &   }. $$
 By hypothesis $M$ has a Rad-supplement in $E(M)$, say $K$, that is, $M + K = E(M)$ and $M\cap K
  \subseteq \Rad K$.
 Since  $M \subseteq h(N)$, we obtain that $h(N) = h(N) \cap E(M) = h(N) \cap (M+K)= M + h(N)\cap K$.
 Now, taking any $n\in N$,  we have $h(n)= m + h(n_1) = h(m + n_1)$ where $m\in M$ and $h(n_1) \in h(N)\cap K$. So, $n= m + n_1 \in M + h^{-1}(K)$ since $h$ is monic, and thus $M + h^{-1}(K)= N$.
  Moreover, $M\cap h^{-1}(K)= h^{-1}(M\cap K)\subseteq h^{-1}(\Rad K)\subseteq \Rad (h^{-1}(K))$ since  $h^{-1}(M) =  M$ as $h$ is monic. Hence
     $h^{-1}(K)$ is a Rad-supplement of $M$ in $N$.
\end{proof}

\begin{proposition}\label{prop:if$A$and$B/A$are-Rad-supplementing-modulesThen-so-is-$B$}
Let $B$ be a module and $A$ be a submodule of $B$. If $A$ and $B/A$ are Rad-supplementing, then so is $B$.
\end{proposition}
\begin{proof}
Let $B\subseteq N$ be any extension of $B$. By hypothesis, there is a Rad-supplement $V/A$ of $B/A$ in $N/A$ and a Rad-supplement $W$ of $A$ in $V$. We claim that $W$ is a Rad-supplement of $B$ in $N$. We have epimorphisms $f: W \to V/A$ and $g:V/A \to N/B$ such that $\Ker f = W\cap A \subseteq \Rad W$ and $\Ker g = V/A \cap B/A \subseteq \Rad (V/A)$. Then $g \circ f : W\to N/B$ is an epimorphism such that $W\cap B = \Ker (g\circ f) \subseteq \Rad W$ by \cite[Lemma 1.1]{Xue:CharacterizationofSemiperfectandperfectrings}. Finally, $N = V + B = (W + A) + B = W + B$.
\end{proof}
\begin{remark}
The previous result holds for supplementing modules; see \cite[Lemma 1.3-(c)]{Zoschinger:ModulnDieInJederErweiterungEinKomplementHaben}.
\end{remark}

\begin{corollary}
If $M_1$ and $M_2$ are Rad-supplementing modules, then so is $M_1 \oplus M_2$.
\end{corollary}
\begin{proof}
Consider the short exact sequence $$0 \to M_1 \to M_1 \oplus M_2 \to M_2 \to 0 .$$ Thus the result follows by Proposition \ref{prop:if$A$and$B/A$are-Rad-supplementing-modulesThen-so-is-$B$}.
\end{proof}

Recall that $R$ is said to be a \emph{left hereditary} ring if every left ideal of $R$ is projective.
\begin{corollary}\label{cor:M/P(M)-rad-suppl-THEN-M-rad-suppl}
If $M/P(M)$ is Rad-supplementing, then $M$ is Rad-supplementing. For left hereditary rings, the converse is also true.
\end{corollary}

\begin{proof}
Since $P(M)$ is Rad-supplementing by Corollary \ref{cor:P(M)-Rad-supplementing}, the result follows by Proposition \ref{prop:if$A$and$B/A$are-Rad-supplementing-modulesThen-so-is-$B$}. Over left hereditary rings, any factor module of a Rad-supplementing module is Rad-supplementing (see Corollary \ref{cor:over-left-heresitary-rings-factor-module-Rad-supplementing}).
\end{proof}

We give the proof of the following known fact for completeness.
\begin{lemma}\label{prop:everySimpleModule-DIRECT-summand-OR-small}
Every simple submodule $S$ of a module $M$ is either a direct summand of $M$ or small in $M$.
\end{lemma}

\begin{proof}
Suppose that $S$ is not small in $M$, then there exists a proper submodule $K$ of $M$ such that $S + K = M$. Since $S$ is simple and  $K \neq M$,  $S\cap K = 0$. Thus $M = S \oplus K$.
\end{proof}

\begin{proposition}\label{cor:every-SIMPLE-module-issupplementing}
Every simple module is (Rad-)supplementing.
\end{proposition}
\begin{proof}
Let $S$ be a simple module and  $N$ be any extension of $S$. Then by Lemma \ref{prop:everySimpleModule-DIRECT-summand-OR-small}, $S \ll N $ or $S \oplus S'= N$ for a submodule $S' \subseteq N$. In the first case, $N$ is a (Rad-)supplement of $S$ in $N$, and in the second case, $S'$ is a (Rad-)supplement of $S$ in $N$. So, in each case $S$ has a (Rad-)supplement in $N$, that is, $S$ is (Rad-)supplementing.
\end{proof}

\begin{theorem}\label{thm:every-module-with-composition-series-is-rad-supp.}
Every module with composition series is (Rad-)supplementing.
\end{theorem}

\begin{proof}
Let $0=M_0 \subseteq M_1 \subseteq M_2 \subseteq \cdots \subseteq M_n = M$ be a composition series of a module $M$. The proof is by induction on $n \in \mathbb{N}$.  If $n=1$, then $M= M_1$ is simple, and so $M$ is (Rad-)supplementing by Proposition \ref{cor:every-SIMPLE-module-issupplementing}. Suppose that this is true for each $k \leq n-1$. Then $M_{n-1}$ is (Rad-)supplementing. Since $M_n/M_{n-1}$ is also (Rad-)supplementing as a simple module, we obtain by Proposition \ref{prop:if$A$and$B/A$are-Rad-supplementing-modulesThen-so-is-$B$} that $M = M_n$ is (Rad-)supplementing.
\end{proof}

\begin{corollary}
A finitely generated semisimple module is (Rad-)supplementing.
\end{corollary}

In general, a factor module of a Rad-supplementing module need not be
Rad-supplementing. To give such a counterexample we need the following result.

Recall that a ring $R$ is called \emph{Von Neumann regular} if every element $a\in R$ can be written in the form $a x a$, for some $x\in R$.
\begin{proposition}\label{prop:Von-neum.regular-Inj=Rad-supplementing}
Let $R$ be a commutative Von Neumann regular ring. Then an $R$-module  $M$ is Rad-supplementing  if and only if $M$ is injective.
\end{proposition}
\begin{proof}
Suppose that $M$ is a Rad-supplementing module. Let  $M \subseteq N$ be any extension of $M$. Then there is a Rad-supplement $V$
of $M$ in $N$, that is, $V + M = N$ and $V \cap M \subseteq \Rad V$. Since all $R$-modules have zero radical by \cite[3.73 and 3.75]{Lam:LecturesOnModulesAndRings}, we have $\Rad V= 0$, and so  $N = V \oplus M$. Conversely, if $M$ is injective and $M \subseteq N$ is any extension of $M$, then $N = M \oplus
K$ for some submodule $K \subseteq N$. Thus $K$ is a Rad-supplement of $M$ in $N$.
\end{proof}

 It is known that a ring $R$ is lefty hereditary if and only if every quotient of an injective $R$-module is injective (see \cite[Ch.I, Theorem 5.4]{Cartan-Eilenberg}).
\begin{example}\label{exm2}
Let $\displaystyle R = \prod _{i\in I} F_i$ be a ring, where each $F_i$
is a field for an infinite index set $I$. Then $R$ is a commutative Von Neumann regular ring. Indeed, let $a = (a_i)_{i\in I}\in R$ where $a_i \in F_i $ for all $i \in I$. Taking $b = (b_i)_{i\in I}\in R$ where $b_i \in F_i $ such that
$$
b_i = \left\{ \begin{array}{ll}
a_i^{-1} & \text{ if $a_i \neq 0)$}\\
0 & \text{ if $a_i =0$}\\
\end{array} \right.
$$
Then we obtain that $$aba = (a_i)_I (b_i)_I (a_i)_I = (a_i b_i a_i)_{i\in I} = (a_i)_{i\in I} = a .$$
Now, by Proposition \ref{prop:Von-neum.regular-Inj=Rad-supplementing}, $R$ is a Rad-supplementing module over itself since it is injective (see \cite[Corollary 3.11B]{Lam:LecturesOnModulesAndRings}). Since $R$ is not  noetherian, it cannot be semisimple (by \cite[Corollary 2.6]{Lam:AFirstCourseInNoncommutativeRings}). Thus $R$ is not hereditary by \cite[Corollary]{Osofsky:Rings-All-Of-whoseFG-modulesInjective}. Hence, there is a factor module of $R$ which is not injective.
\end{example}

The following technical lemma will be useful to show that Rad-supplementing modules are closed under factor modules, under a special condition.

\begin{lemma}\label{lemma:if C/A injectve then we have the following commutative diagram}
Let $A \subseteq B \subseteq C$ be modules with $C/A$ injective. Let $N$ be a module containing $B/A$. Then there exists a commutative diagram with exact rows:
$$\xymatrix {0 \ar[r] & A  \ar@{^{(}->}[r]  \ar[d]_{id}  &  B \ar[r]  \ar[d] &  B/A \ar[r]  \ar@{^{(}->}[d] &  0 \\
0 \ar[r] & A \ar[r]  &  P \ar[r]  &  N \ar[r]  &  0 } .$$
\end{lemma}

\begin{proof}
By pushout we have the following commutative diagram, where $\varphi$ exists since $C/A$ is injective:
$$
\xymatrix{0 \ar[r] & B/A \ar@{^{(}->}[r] \ar@{^{(}->}[d] & N \ar[r] \ar[d]^g \ar[dl]_{(1)} | \varphi &  N/(B/A) \ar[r] \ar[d]^{id} \ar@{-->}[dl]_{\alpha}^{(2)}& 0 \\
0 \ar[r] & C/A \ar[r]^{\beta} & N' \ar[r] & N/(B/A) \ar[r] &  0 }
$$
In the diagram, since the triangle-(1) is commutative, there exists a homomorphism $\alpha: N/(B/A) \To N'$ making the triangle-(2) is commutative by \cite[Lemma I.8.4]{FuchsAndSalce:ModulesOverNonNoetherianDomains}. So, the second row splits. Then we can take $N' = (C/A) \oplus (N/(B/A))$, and so we may assume that $\beta : C/A \To N'$ is an inclusion. Therefore, we have the following commutative diagram since $B/A = \beta(B/A) = g(B/A) \subseteq N'$:
$$
\xymatrix{0 \ar[r] & A \ar@{^{(}->}[r] \ar[d]_{id} & B \ar[r] \ar[d]_{\phi}  & B/A \ar[r] \ar@{^{(}->}[d] & 0 \\
0 \ar[r] & A \ar[r]^-\gamma  & \quad C\oplus (N/(B/A)) \quad  \ar[r]^-\sigma & N' \ar[r] &  0 }
$$
where $\gamma (a) = (a, 0)$ for every $a\in A$, $\phi (b) = (b, 0)$ for every $b\in B$, and $\sigma (c, \overline{x}) = (c+ A, \overline{x})$ for every $c\in C$ and $\overline{x}\in N/(B/A)$.
Finally, taking $P = \sigma ^{-1} (g(N))$ and defining a homomorphism $\widetilde{\sigma}: P \To g(N)$ by  $\widetilde{\sigma} (x) = \sigma (x)$ for every $x\in P$ (in fact, $\widetilde{\sigma}$ is an epimorphism as so is $\sigma$),  we obtain the following desired commutative diagram:
$$
\xymatrix{0 \ar[r] & A \ar@{^{(}->}[r] \ar[d]_{id} & B \ar[r] \ar[d]  & B/A \ar[r] \ar[d] & 0 \\
0 \ar[r] & A \ar[r]  & P  \ar[r]^-{\widetilde{\sigma}} & g(N)\cong N \ar[r] &  0 }
$$
\end{proof}

\begin{proposition}
Let $A\subseteq B \subseteq C$ with $C/A$ injective. If $B$ is
Rad-supplementing, then so is $B/A$.
\end{proposition}

\begin{proof}
Let $B/A \subseteq N$ be any extension of $B/A$. By Lemma \ref{lemma:if C/A injectve then we
have the following commutative diagram}, we have the following
commutative diagram with exact rows since $C/A$ is injective:
$$
\xymatrix {0 \ar[r] & A  \ar@{^{(}->}[r]  \ar[d]^{id}  &  B \ar[r]^{\sigma}  \ar[d]^h &  B/A \ar[r]  \ar@{^{(}->}[d]^f &  0 \\
0 \ar[r] & A \ar[r]  &  P \ar[r]^g   &  N \ar[r]  &  0 } .
$$
Since $h$ is monic and $B$ is Rad-supplementing, $B\cong \Image h$ has a Rad-supplement
 in $P$, say $V$, that is, $\Image h + V = P$ and $\Image h \cap V \subseteq
\Rad V$. We claim that $g (V)$ is a Rad-supplement of $B/A$ in $N$.
$$
 N = g(P) = g(h(B)) + g(V) = (f \sigma) (B) + g(V) =
(B/A) + g(V), \text{ and}
$$
$$(B/A)\cap g(V) = f(\sigma (B))\cap
g(V) = g[h(B) \cap V] \subseteq g(\Rad V)
\subseteq \Rad (g(V)).
$$
\end{proof}

\begin{corollary}\label{cor:over-left-heresitary-rings-factor-module-Rad-supplementing}
If $R$ is a left hereditary ring, then every factor module of Rad-supplementing module is Rad-supplementing.
\end{corollary}

\begin{proposition}\label{prop:If-M-reduced-projective-Rad-supplementing-then-It-has-small-radical}
If $M$ is a reduced, projective and Rad-supplementing module, then $\Rad M \ll M$.
\end{proposition}
\begin{proof}
Suppose $X+\Rad M = M$ for a submodule $X$ of $M$. Then since $M$ is projective,  there exists $f\in \End (M)$ such that $\Image f \subseteq X$ and $\Image (1-f) \subseteq \Rad M = J M$ where $J$ is a Jacobson radical of $R$. Therefore $f$ is a monomorphism by \cite[Theorem 3]{Beck:Projective-and-free-modules}. Since $M$ is Rad-supplementing and $\Image f \cong M$,  $\Image f$ has a Rad-supplement $V$ in $M$, that is, $\Image f + V = M$ and $\Image f \cap V \subseteq \Rad V$. Now we have an epimorphism $g: V \to M /\Image f$ such that $\Ker g = V \cap \Image f \subseteq \Rad V$. Moreover, since $M = \Image f + \Image (1-f) = \Image f + \Rad M$ we have $\Rad (M/\Image f) = M/ \Image f$. Thus $\Rad V = V$, and so $V=0$ since $M$ is reduced. Hence $M=\Image f \subseteq X$ implies that $X=M$ as required.
\end{proof}

Recall that $R$ is said to be a \emph{semilocal} ring if $R/J$ is a semisimple ring, that is a left (and right) semisimple $R$-module (see \cite[\S 20]{Lam:AFirstCourseInNoncommutativeRings}).

\begin{theorem}
A ring $R$ is left perfect if and only if $R$ is semilocal, $_R R$ is reduced and the free left $R$-module $F=(_R R)^{(\mathbb{N})}$ is Rad-supplementing.
\end{theorem}
\begin{proof}
If $R$ is left perfect, then $R$ is semilocal by \cite[28.4]{Anderson-Fuller:RingsandCategoriesofModules}, and clearly $_R R$ is reduced.  Since all left $R$-modules are supplemented and so Rad-supplemented, $F$ is  Rad-supplementing. Conversely, since $P(_R R)= 0$ we have  $P(F)=(P(_R R))^{(\mathbb{N})} = 0$, that is, $F$ is reduced. Thus by Proposition \ref{prop:If-M-reduced-projective-Rad-supplementing-then-It-has-small-radical}, $J F = \Rad F \ll F$, that is, $J$ is left $T$-nilpotent by, for example, \cite[28.3]{Anderson-Fuller:RingsandCategoriesofModules}. Hence $R$ is left perfect by \cite[28.4]{Anderson-Fuller:RingsandCategoriesofModules} since it is moreover semilocal.
\end{proof}

\cite{Averdunk:Ph.D.thesis} has studied supplementing modules over commutative noetherian rings, and he showed that If a module $M$ is supplementing, then it is \emph{cotorsion}, that is, $\Ext^1_R (F, M) = 0$ for every flat module $F$ (see \cite{Enochs-Jenda:RelativeHomologicalAlgebra} for cotorsion modules). So the question was raised When Rad-supplementing modules are cotorsion? Since any pure-injective module is cotorsion, the following result gives an answer of the question for a semisimple module over a commutative ring. The relation between (Rad-)supplementing modules and cotorsion modules needs to be further investigated.

The part (iii)$\Rightarrow$(i) of the proof of the following theorem is the same as the proof given  in \cite[Theorem 1.6-(ii)$\Rightarrow$(i)]{Zoschinger:ModulnDieInJederErweiterungEinKomplementHaben}, but we give it by explanation for completeness.
\begin{theorem}\label{thm:supplementing=rad-supplementing=pure-injective}
Let $R$ be a commutative ring. Then the following are equivalent for a semisimple $R$-module $M$.
\begin{enumerate}
\item[(i)] $M$ is supplementing,
  \item[(ii)] $M$ is Rad-supplementing,
  \item[(iii)] $M$ is pure-injective,
\end{enumerate}
\end{theorem}
\begin{proof}
 (i)$\Rightarrow$(ii) is clear. (ii)$\Rightarrow$(iii) Let $M \subseteq N$ be a pure extension of $M$. By hypothesis $M$ has a Rad-supplement $V$ in $N$, that is, $M + V = N$ and $M\cap V \subseteq \Rad V$. Since $M$ is pure in $N$, we have $\Rad M = M\cap \Rad N$ (as $R$ is commutative). Thus $M\cap V \subseteq M \cap \Rad N = \Rad M = 0$ as $M$ is semisimple. Hence $N=M\oplus V$ as required. (iii)$\Rightarrow$(i) Let $M\subseteq N$ be any extension of $M$. Then the factor module $X=(M+\Rad N)/\Rad N$ of $M$ is again semisimple and pure-injective. Since semisimple submodules are pure in every module with zero radical and $\Rad (N/\Rad N) = 0$, it follows that $X$ is a direct summand of $N/\Rad N$. Now let $$(V/\Rad N) \oplus X = N/ \Rad N$$ for a submodule $V \subseteq N$ such that $\Rad N \subseteq V$. So we have $V+M = N$ with $V$ minimal, and thus $V$ is a supplement of $M$ in $N$. This is because, if $T+M =N$ for a submodule $T$ of $N$ with $T\subseteq V$, then from $$\Rad (N/T) = \Rad ((M+T)/T) = \Rad (M/M\cap T) = 0 $$ as $M/M\cap T$ is semisimple,
 we obtain that $\Rad N \subseteq T$. Moreover, since $$\Rad N = V\cap (M+\Rad N) = V\cap M + \Rad N ,$$
 we have $V\cap M \subseteq \Rad N$ and
 $V = T + V\cap M \subseteq T + \Rad N = T$, thus $T=V$. 
\end{proof}

\section{Ample Rad-supplementing modules}\label{sec:Ample-rad-supp}

The following useful result gives a relation between Rad-supplementing modules and ample Rad-supplementing modules.
\begin{proposition}\label{prop:M-ample-rad-supplementing-iff-every-submodule-rad-supplementing}
A module $M$ is ample Rad-supplementing if and only if every
submodule of $M$ is Rad-supplementing.
\end{proposition}

\begin{proof}
$(\Leftarrow)$ Let $M$ be a module and $N$ be any extension of $M$.
Suppose that for a submodule $X\subseteq N$, $X + M = N$. By
hypothesis the submodule $X\cap M$ of $M$ has a Rad-supplement $V$
in $X$ containing $X \cap M$, that is, $(X \cap M) + V = X$ and
$(X \cap M) \cap V \subseteq \Rad V$. Then $N = M + X = M+(X \cap
M) + V = M + V $ and, $M \cap V = M\cap (V \cap X) = (X \cap M) \cap V
\subseteq \Rad V$. Hence $V$ is a Rad-supplement of $M$ in $N$ such that $V\subseteq X$.

$(\Rightarrow)$ Let $U$ be a submodule of $M$ and $N$ be any module
containing $U$. Thus we can draw the pushout for the
inclusion homomorphisms $i_1 :U\hookrightarrow N$ and  $i_2 :U
\hookrightarrow M$:
$$ \xymatrix{
   M \ar@{-->}[r]^{\alpha} & F  \\
   U \,\ar@{^{(}->}[r]_{i_1} \ar@{^{(}->}[u]^{i_2} & N \ar@{-->}[u]_{\beta}
} .$$
In the diagram, $\alpha$ and $\beta$ are also monomorphisms by the
properties of pushout (see, for example, \cite[Exercise 5.10]{Rotman(2009):AnIntroductionToHomologicalAlgebra}). Let $M'= \Image \alpha$ and $N'
= \Image \beta$. Then $F = M' + N'$ by the properties of pushout. So by
hypothesis, $M' \cong M$ has a Rad-supplement $V$ in $F$
 such that $V\subseteq N'$, that is,  $M' + V = F$ and $M'\cap V
\subseteq \Rad V$. Therefore $V$ is a Rad-supplement of $M' \cap N'$
in $N'$, because $N' = N' \cap F = N' \cap (M' + V) = (M' \cap N')+V$
and $(M' \cap N')\cap V = M'\cap V \subseteq \Rad V$. Now, we claim that
$\beta^{-1}(V)$ is a Rad-supplement of $U$ in
$N$. Since $\beta : N \rightarrow F$ is a monomorphism with $N' =
\Image \beta$, we have an isomorphism $\widetilde{\beta} : N\rightarrow N'$ defined as
$\widetilde{\beta} (x)= \beta (x)$ for all $x\in N$. By this isomorphism, since $V$ is a Rad-supplement of
$M' \cap N'$ in $N'$, we obtain  $\widetilde{\beta}^{-1}(V)$ is
a Rad-supplement of $\widetilde{\beta}^{-1}(M'\cap N')$ in
$\widetilde{\beta}^{-1}(N')$. Since it can be easily shown that $\widetilde{\beta}^{-1}
(V)=\beta ^{-1}(V)$, $\widetilde{\beta}^{-1}(N') = N$, and $\widetilde{\beta}^{-1}(M'\cap N') = U$ the result follows.
\end{proof}

\begin{corollary}\label{cor:every-ample-rad-supplementing-module-is-rad-supplemented}
Every ample Rad-supplementing module is both Rad-supplementing and Rad-supplemented.
\end{corollary}

\begin{theorem}\label{thm:R-perfect-iff-R-reduced-and-F-ample-rad-supp}
A ring $R$ is left perfect if and only if $_R R$ is reduced and the free left $R$-module $F = (_R R)^{(\mathbb{N})}$ is ample Rad-supplementing.
\end{theorem}
\begin{proof}
If $R$ is left perfect, then $_R R$ is reduced
  and all left $R$-modules are supplemented, and so Rad-supplemented. Thus every submodule of $F$ is Rad-supplementing. Hence $F$ is ample Rad-supplementing by Proposition \ref{prop:M-ample-rad-supplementing-iff-every-submodule-rad-supplementing}. Conversely, if $F$ is ample Rad-supplementing, then it is Rad-supplemented by Corollary \ref{cor:every-ample-rad-supplementing-module-is-rad-supplemented}, and so $R$ is left perfect by  \cite[Theorem 5.3]{Rad-supplementedModules}.
\end{proof}

Finally, we give the characterization of the rings over which every module is (ample) Rad-supplementing.
\begin{theorem}\label{thm:characterizaiton-of-rad-supp}
For a ring $R$, the following are equivalent:
\begin{enumerate}
  \item Every left $R$-module is Rad-supplementing;
  \item Every reduced left $R$-module is Rad-supplementing;
  \item Every left $R$-module is ample Rad-supplementing;
  \item Every left $R$-module is Rad-supplemented;
  \item $R/P(R)$ is left perfect.
\end{enumerate}
\end{theorem}
\begin{proof}Let $M$ be a module.
(1)$\Rightarrow$(2) is clear. (2)$\Rightarrow$(1) Since $M/P(M)$ is reduced, it is Rad-supplementing by hypothesis. So $M$ is Rad-supplementing by Corollary \ref{cor:M/P(M)-rad-suppl-THEN-M-rad-suppl}. (1)$\Rightarrow$(3)  Since every submodule of $M$ is Rad-supplementing, $M$ is ample Rad-supplementing by Proposition \ref{prop:M-ample-rad-supplementing-iff-every-submodule-rad-supplementing}. (3)$\Rightarrow$(4) by Corollary \ref{cor:every-ample-rad-supplementing-module-is-rad-supplemented}. (4)$\Rightarrow$(1) Let $M\subseteq N$ be any extension of  $M$. By hypothesis, $N$ is Rad-supplemented, and so $M$ has a Rad-supplement in $N$. (4)$\Leftrightarrow$(5) by \cite[Theorem 6.1]{Rad-supplementedModules}.
\end{proof}

--------------------------------------------------

\end{document}